\documentclass{article}

\usepackage{arxiv}

\usepackage{amsthm}

{
      \theoremstyle{plain}
      \newtheorem{assumption}{Assumption}

      \theoremstyle{plain}
      \newtheorem{algo}{Algorithm}

      \theoremstyle{plain}
      \newtheorem{definition}{Definition}

     \theoremstyle{plain}
      \newtheorem{lemma}{Lemma}

      \theoremstyle{plain}
      \newtheorem{corollary}{corollary}

      \theoremstyle{plain}
      \newtheorem{theorem}{Theorem}
  }
  
\usepackage[utf8]{inputenc} 
\usepackage[T1]{fontenc}    
\usepackage{hyperref}       
\usepackage{url}            
\usepackage{booktabs}       
\usepackage{amsfonts}       
\usepackage{nicefrac}       
\usepackage{microtype}      
\usepackage{lipsum}		
\usepackage{graphicx}
\usepackage{natbib}
\usepackage{doi}

\usepackage{amsmath}

\usepackage{newtxtext}
\usepackage[subscriptcorrection]{newtxmath}
\usepackage{comment}
\usepackage{multirow}
\usepackage{enumitem}
\graphicspath{{./art/}}

\usepackage[plain,noend]{algorithm2e}

\makeatletter
\renewcommand{\algocf@captiontext}[2]{#1\algocf@typo. \AlCapFnt{}#2} 
\def\@algocf@capt@plain{top}
\renewcommand{\algocf@makecaption}[2]{%
  \addtolength{\hsize}{\algomargin}%
  \sbox\@tempboxa{\algocf@captiontext{#1}{#2}}%
  \ifdim\wd\@tempboxa >\hsize
    \hskip .5\algomargin%
    \parbox[t]{\hsize}{\algocf@captiontext{#1}{#2}}
  \else%
    \global\@minipagefalse%
    \hbox to\hsize{\box\@tempboxa}
  \fi%
  \addtolength{\hsize}{-\algomargin}%
}
\makeatother

\def\Bka{{\it Biometrika}}
\def\AIC{\textsc{aic}}
\def\T{{ \mathrm{\scriptscriptstyle T} }}
\def\v{{\varepsilon}}

\begin{document}

\markboth{M. Sabbagh and D.A. Stephens}{Semi-parametric Bayesian inference}

\title{Semi-parametric Bayesian inference under Neyman orthogonality}

\author{ \hspace{1mm}Magid Sabbagh\thanks{Corresponding author}\\
	Department of Mathematics and Statistics\\
	McGill University\\
	Montreal, QC, Canada \\
	\texttt{magid.sabbagh@mail.mcgill.ca} \\
	\And
	\href{https://orcid.org/0000-0001-9811-7140}{\includegraphics[scale=0.06]{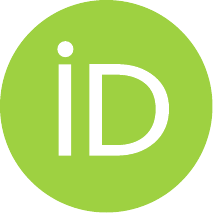}\hspace{1mm}David A. Stephens} \\
	Department of Mathematics and Statistics\\
	McGill University\\
	Montreal, QC, Canada \\
	\texttt{david.stephens@mcgill.ca} \\
}


\maketitle

\begin{abstract}
The validity of two-step or plug-in inference methods is questioned in the Bayesian framework. We study semi-parametric models where the plug-in of a non-parametrically modelled nuisance component is used. We show that when the nuisance and targeted parameters satisfy a Neyman orthogonal score property, the approach of cutting feedback through a two-step procedure is a valid way of conducting Bayesian inference. Our method relies on a non-parametric Bayesian formulation based on the Dirichlet process and the Bayesian bootstrap. We show that the marginal posterior of the targeted parameter exhibits good frequentist properties despite not accounting for the inferential uncertainty of the nuisance parameter.  We adopt this approach in Bayesian causal inference problems where the nuisance propensity score model is estimated to obtain marginal inference for the treatment effect parameter, and demonstrate that a plug-in of the propensity score has a negligible effect on marginal posterior inference for the causal contrast. We investigate the absence of Neyman orthogonality and exploit our findings to show that in conventional two-step procedures, the posterior distribution converges under weaker restrictions than those needed in the frequentist sequel. For a simple family of useful scores, we demonstrate that even in the absence of Neyman orthogonality, the posterior distribution is asymptotically unchanged by the estimation of the nuisance parameter, merely provided the latter estimator is consistent.

\end{abstract}


\section{Introduction}{\label{Introduction}}

Suppose that the true value $\theta_0 \in \Theta \subset \mathbb{R}^p$ of a parameter $\theta$ in a statistical model minimizes the expected loss
\[
\theta_0  \equiv  \text{argmin}_{\theta}\text{ }E_{P_O} l(O;\theta).
\]
for some loss function $l :  \mathcal{O} \times \Theta \to [0,\infty)$, where measure $P_O$, with distribution $F_O$, is the data generating model. We seek to perform Bayesian inference for $\theta_0$.  If $P_O$ is unknown but is assumed to belong to a set $\mathcal{F}$, any prior $P_{\mathcal{F}}$ on $\mathcal{F}$ induces a prior on $P_\Theta$ on  $\Theta$, by considering, for $F \in \mathcal{F},$
\[
\theta(F)= \text{argmin}_{\theta} \int_{\mathcal{O}} l(o;\theta)dF(o)
\]
and then defining $P_\Theta (\theta \in A) = P_{\mathcal{F}}\left\{F: \theta(F) \in A \right\}$.  Even if $P_{\mathcal{F}}$ is a non-parametric prior, that is, no finite-dimensional parameter characterizes $F_O$ precisely, it is still the case that the induced prior on $\theta_0$ is well-defined.  However, in general $\theta_0$ does not characterize $F_O$, so standard approaches to inference using a parametric likelihood cannot be used.

Semi-parametric Bayesian inference is necessary in situations of partial specification where components of the model, such as moments, are parametrized, but the remainder of the data-generating mechanism is modeled non-parametrically.  We focus on applications in causal inference as our motivation.  Semi-parametric models are common in causal inference, where the inference task is to infer effect on an intervention (a treatment for example) on an outcome subject to confounding by other measured variables \citep{Rubin:1977, Rubin:1979, Pearl:2009}.  In this setting, the variation in the outcome as a function of the treatment is specified only via a mean model whose form is deemed too complex to represent standard approaches such as regression.  The propensity score, a function that encapsulates the treatment assignment model, has played a key role in adjusting for confounding in observational studies  \citep{Rosenbaum/Rubin:1983,Rosenbaum/Rubin:1984}. In most cases, the treatment assignment model needs to be estimated from the observed data, with the fitted values used to carry out adjustment. 

The use of fitted values of the propensity score in the classical frequentist literature is common, but in the Bayesian sequel it is still controversial.  For example, it is debated whether valid Bayesian inference using this plug-in is possible. \citet{Li/etal:2022} survey methods in Bayesian causal inference and discusses thoroughly the advantages and drawbacks of the methods generally deployed, such as the cutting feedback \citep{Mccandless/etal:2009} and two-step approaches \citep{Kaplan/Chen:2012}. The Linked Bayesian Bootstrap \citep{Stephens/etal:2022} overcomes the difficulties found in the previous approaches. It relies on a Bayesian non-parametric formulation and puts priors on the distributions rather than finite dimensional parameters. 
The non-parametric prior-to-posterior update provides a way of obtaining a sample of the posterior of the treatment effect, a notion that we will revisit in section \ref{TargetingPoI}. This procedure provides the correct way of handling the inferential uncertainty of the fitted propensity score, yielding good frequentist properties.

Central to the arguments in this paper is the concept of Neyman orthogonality.  If $\theta_0$ and $h_0$ are respectively the target and nuisance parameters satisfying the moment restriction $E_{P_O}\left\{m(O;\theta_0,h_0)\right\}=0$, we say that the score $m$ is Neyman orthogonal if
\[
 \left.  \frac{d}{dt} E_{P_O}\left[ m\left\{O;\theta_0,h_0+t(h-h_0)\right\}\right] \right |_{t=0}=0
\]
for all possible $h$. Neyman orthogonality, which we revisit and properly define in section \ref{NO}, has been the foundation of the seminal work of \citet{Newey:1994} and \citet{Chernozhukov/etal:2018}.  Although these works differ in imposing nuisance parameter restrictions, both have exploited Neyman orthogonality in reducing the impact for estimating the nuisance parameters in the marginal frequentist inference of the parameter of interest. In our motivating setting, many quantities are based on Neyman orthogonal scores;  for example the average treatment effect, the average treatment effect on the treated, and the local average treatment effect parameters satisfy Neyman orthogonal scores. Such scores are often linear when considered as functions of $\theta \in \Theta$. 

In the Bayesian setting, we show that under Neyman orthogonality, one may avoid incorporating Bayesian inference for the nuisance parameters and work with frequentist plug-in estimates. Recently, \citet{Yiu/etal:2025} used efficient influence functions to obtain posterior corrections of one-step estimators based on the Bayesian bootstrap. In fact, these corrections induce Neyman orthogonality; see \citet{Chernozhukov/etal:2018} for a relation between influence functions and Neyman orthogonal scores. The approach we adopt is based on targeting via an estimating equation, and on a Bayesian/frequentist duality statement.
Furthermore, we capitalize on the Bayesian/frequentist duality to establish that convergence of posterior distribution still holds even when Neyman orthogonality is not met, and we validate the fact that the nuisance parameter, provided it is estimated consistently, does not asymptotically impact the posterior distribution, irrespective of its convergence rate. 

\section{Targeting Parameters of Interest}{\label{TargetingPoI}}
\subsection{Notation}
We adopt the following notation:  $O_1,\ldots,O_n$ are i.i.d random variables defined on a probability space $(\mathcal{O},B)$ with distribution $P_O$. Define the empirical process $P_n$ and bootstrapped empirical process $P_n^w$ by
\[
P_n=\frac{1}{n}\sum\limits_{i=1}^n \delta_{O_i} \qquad \text{and} \qquad P_n^w=\sum\limits_{i=1}^n w_{in}\delta_{O_i}
\]
respectively, where $(w_{1n}, \ldots,w_{nn})$ are random weights. When needed and mainly for notation convenience,  we define $P: L^{1}(\mathcal{O},\mathcal{B},P_O) \to R$, by 
\[
Pf:=E_{P_O}\left\{f(O)\right\}=\int_O f(o)dP_O(o).
\]
In the rest of the paper, weights drawn from a distribution $P_W$ are assumed to be independent of the data. The joint measure is hence $P_{OW}=P_{O}^\infty \times P_W$. We write $P_O$ instead of $P_O^{\infty}$ for simplicity. The relevant probability spaces are properly defined in the Supplementary Material. The theory developed in section \ref{MainResults} can be extended to other families of weights \citep{Praestgaard/Wellner:1993}, including Efron's bootstrap \citep{Efron:1970} and the double bootstrap. In fact, $P_W$ can be the distribution of a vector $\mathbf{W}_n=\left(w_{1n},\ldots,w_{nn}\right)$ with exchangeable non-negative components summing to $1$ with the following properties:
\begin{enumerate}
\item $\lim\limits_{x \to \infty} \limsup\limits_{n \to \infty} \sup\limits_{x \geq \lambda} x^2 P(nw_{1n}> x) =0$,  
\item $\limsup\limits_{n \to \infty} n \Vert w_{1n} \Vert_{2,1} \leq C$ for some $C>0$, where $\Vert w_{1n} \Vert_{2,1}= \int_0^1 \sqrt{P(w_{1n}>t})dt$
 \item $n^{-1} \sum_{i=1}^n (n w_{in}-1)^2 = c^2+o_{P_W}(1)$ for some $c>0$
\end{enumerate}
In this paper, we suppose $\mathbf{W}_n\sim \text{Dirichlet}(1,\ldots,1)$ (so that $c=1$) due to the Bayesian interpretability of the results.  
\subsection{Posterior Distribution of Targeted Parameters}
The approach to targeting parameters that involves expressing them as minimizers of loss functions have been considered as partial specifications by \citet{Bissiri/etal:2016} in the absence of likelihoods.  \citet{Lyddon/etal:2019}, who generalized the work of \citet{Newton/Raftery:1994}, obtain a sample from the posterior distribution of $\theta(P_O)$ using a Bayesian non-parametric computational approach. Such strategy relies on the Dirichlet process \citep{Ferguson:1973}, acting as a prior on distribution functions. Suppose that $\alpha$ is a finite Borel measure on $(\mathcal{O},B)$. If a Dirichlet process prior DP$(\alpha)$ is placed on $F$, then the posterior distribution of $F \mid O_1,\ldots,O_n$ will follow a Dirichlet process  DP$(\alpha+nP_n)$. As $\alpha(\mathcal{O}) \to 0$, the posterior distribution can represented by the empirical process $P_n^w$ with $(w_{1n},\ldots,w_{nn}) \sim \text{Dir}(1,\ldots,1)$. A sample from the posterior of the target parameter can hence be obtained by solving
\[
\text{argmin}_\theta \int_{\mathcal{O}} l(O;\theta) \ dP_n^w = \text{argmin}_\theta \sum\limits_{i=1}^n w_{in}l(O_i;\theta).
\]
A similar argument can be constructed if $\theta_0 \equiv \theta_0(P_O)$ is the solution to a moment restriction $E_{P_O} m(O;\theta)=0.$  A posterior for the targeted parameter can be obtained by solving
\[
 \int_{\mathcal{O}} m(O;\theta) \  dP_n^w  \equiv \sum\limits_{i=1}^n w_{in} m(O_i;\theta) =0.
\]
These computational strategies hinge on the fact that the targeted parameter satisfies a moment restriction or minimizes a loss function. Such partial specifications provide a method to bypass the knowledge of the full statistical model or data-generating process. Posterior distributions computed via the Bayesian bootstrap enjoy good properties. The following theorem \citep{Kosorok:2008} provides the Bayesian/frequentist duality ensuring the good properties of the posterior computed through the Bayesian bootstrap.
    \begin{theorem}\label{thmllbroot}
    Let $\hat{\theta}_n$ and $\hat{\theta}_{n,BB}$ satisfy 
    \[
    \sum\limits_{i=1}^n m(O_i;\hat{\theta}_n)=0 \qquad \text{and} \qquad  \sum\limits_{i=1}^n w_{in} m(O_i;\hat{\theta}_{n,BB})=0
    \]
    respectively.  Then, under regularity conditions,
\[
   \sqrt{n} ( \hat{\theta}_n-\theta_0 ) 
 \qquad \text{and} \qquad   \sqrt{n} (\hat{\theta}_{n,BB}-\hat{\theta}_n)|O_1,\ldots O_n
   \]
   converge in distribution as $n \to \infty$ to a random variable with a $\mathcal{N}(0, L)$ distribution, where 
   \[
   L= \left[\mathbb{E}_{P_O}\left\{ \frac{\partial{m(O;\theta_0)}}{\partial \theta^\top}\right\}\right]^{-1} \mathbb{E}_{P_O}\left\{ m(O;\theta_0)m(O;\theta_0)^\top \right\}\left[\mathbb{E}_{P_O}\left\{ \frac{\partial{m(O;\theta_0)}}{\partial \theta^\top}\right\}\right]^{-\top}
   \]
\end{theorem}
  The conditional asymptotic distribution of $\sqrt{n}\left(\hat{\theta}_{n,BB}-\theta_n\right) \mid O_1,\ldots,O_n$ from Theorem \ref{thmllbroot} is the same as the limiting distribution of $\sqrt{n} ( \hat{\theta}_n-\theta_0)$. However, the randomness in the unconditional statement comes through the random variables $O_1\ldots,O_n$, while in the conditional one, the observations $O_1,\ldots,O_n$ are fixed, and the only source of randomness is the Dirichlet weights. When the duality between the frequentist statement and its  Bayesian counterpart holds, Bayesian credible regions have asymptotically the correct coverage probability and coincide with the frequentist confidence intervals based on $\hat{\theta}_n$.

\subsection{Nuisance Parameters}{\label{NuiPar}}
In the presence of a nuisance parameter, the above formulations may lack clarity. In the example of section \ref{Introduction}, it is unclear how to incorporate the propensity score in order to infer about the treatment effect. \citet{Stephens/etal:2022} propose a solution to this problem when the nuisance parameter is itself the minimizer of a loss function or a solution to a parametrically specified estimating equation. 

If the propensity score is parametrically modeled, then inference for the nuisance propensity score model can proceed using standard approaches, for example using an estimating equation using the function $u(o;h)$ in addition to the principal estimating function $m(o;\theta,h)$ where $Em(O;\theta_0,h_0)=0$.  In this case \cite{Sabbagh/Stephens:2026} establishes the principal results concerning Bayesian inference, namely
\begin{enumerate}[label=(\roman*)]
    \item there exists a Bayesian/frequentist duality that permits approximate Bayesian inference via a study of the frequentist sequel.  This duality delivers good frequentist properties including asymptotic normality and coverage at the nominal level.
    \item a plug-in method, that relies on simply using an estimate of the propensity model for adjustment, produces an approximately Normal posterior that exhibits consistent recovery of the target parameter.  The plug-in method must be consistent (at the usual parametric rate) for $h_0$. However, this is not the optimal Bayesian solution, even if the plug-in is the true value $h_0$.
    \item a Bayesian approach that correctly propagates uncertainty in the propensity model produces a posterior distribution with lower posterior variance than the posterior computed with a plug-in estimate.  Specifically, the Linked Bayesian Bootstrap that uses a single set of weights and yields the simultaneous system
    \[
\int_{\mathcal{O}} \begin{pmatrix} m(O;\theta, h) \\  u(O;h) \end{pmatrix} \ dP_n^w = \sum\limits_{i=1}^n w_{in}\begin{pmatrix} m(O_i;\theta, h) \\  u(O_i;h) \end{pmatrix}  = \begin{pmatrix} 0 \\ 0\end{pmatrix}
\]
provides the optimal approach based on $m$.
\item in this case, $\theta$ and $h$ are \textit{a posteriori} asymptotically independent.
\end{enumerate}
These results apply in the purely parametric setting, and in the semi-parametric setting in which the nuisance parameter is finite-dimensional.

\section{Non-parametric Assumptions}{\label{DonskerAss}}
\subsection{Motivation}
A question that remains is how to handle the case when no parametric assumptions are made to estimate the nuisance parameter, and the only viable approach is through non-parametric methods. In this case, we must assess how to propagate uncertainty in estimating the nuisance parameter, $h_0$ say, into the Bayesian inference step for the parameter of interest, and how to ensure good frequentist properties. Suppose that the true values $\theta_0 \in \Theta$ of a targeted parameter $\theta $ and $h_0 \in \mathcal{H}$ of the nuisance parameter $h$ satisfy the moment condition $E_{P_O}\left\{m(O;\theta_0,h_0)\right\}$, where $m : \mathcal{O} \times \Theta \times \mathcal{H} \mapsto R^p$. If $h_0$ were known, a posterior for $\theta_0$ can be obtained by solving
\[
\sum\limits_{i=1}^n w_{in}m(O_i,\theta,h_0)=0.
\]
If $h_0$ is unknown, a posterior for $\theta_0$, conditional on an estimator $\hat{h}$ of $h_0$ can be obtained by solving 
\[
\sum\limits_{i=1}^n w_{in}m(O_i,\theta,\hat{h})=0.
\]
As seen in section \ref{NuiPar}, when a parametric model arising from an estimating equation is available for $h_0$, a simple plug-in without propagating the uncertainty is in general not the optimal way to perform Bayesian analysis. However, when no parametric model is available and flexible methods must be used, it is non-trivial to come up with an analogue to the Linked Bayesian Bootstrap. In the absence of a parametric likelihood or estimating equations defining the nuisance parameter, accounting for uncertainty through the Linked Bayesian Bootstrap becomes considerably more difficult. We explore other questions that can motivate our proposed solution to this problem:
\begin{enumerate}
    \item What if our model does not propagate uncertainty  but at the same time is robust to biased but consistent estimation of $h_0$?
    \item In that case, do we get good frequentist properties, and the Bayesian/frequentist duality?
\end{enumerate}
We establish some of the key conditions in the next section.

\subsection{Neyman orthogonality}{\label{NO}}

In the remaining parts of the paper, we suppose that the targeted parameter space $\Theta$ is a subset of $R^p$ and that the nuisance space  $\mathcal{H}$ is a convex subset of a normed space $\mathcal{N}$.For a vector $v \in R^p$, we write $\Vert v \Vert_{p,2}$ for the Euclidean norm of $v$ and for $h \in \mathcal{H},$ we write $ \Vert h \Vert_{\mathcal{H}}$ for the norm of $h$. If $h,h' \in \mathcal{H}$, we use for the sake of notational convenience $\Vert h-h'\Vert_{\mathcal{H}}$ for the norm of $h-h'$ in the ambient normed vector space $\mathcal{N}$. The true values $\theta_0 \in \Theta$ of  $\theta $ and $h_0 \in \mathcal{H}$ of $h$ satisfy the moment restriction $E_{P_O}\left\{m(O;\theta_0,h_0)\right\}$, for some score function $m : \mathcal{O} \times \Theta \times \mathcal{H} \to R^p$.
\begin{definition}{\label{NeymanOrth}}
        We say that the score $m$ is orthogonal with respect to $h$ if $f_{h-h_0}'(0)=0$, 
        where $f_{h-h_0}(t)=E_{P_O}\left[ m\left\{O;\theta_0,h_0+t(h-h_0)\right\}\right]$ for all $h \in \mathcal{H}$, and $f'$ denotes the ordinary first derivative of $f$ with respect to its argument.
\end{definition}
In order to simplify notation, the subscript $h-h_0$ will no longer be written. 
 \subsection{Problem formulation and assumptions}{\label{PFA}}
 Let $\hat{h}$ be an estimator of $h_0$ and denote by $\hat{\theta}_n$ and $\hat{\theta}_{n,BB}$ the respective solutions to 
 \[
     \sum\limits_{i=1}^n m(O_i,\hat{\theta}_{n},\hat{h})=0 \quad \text{ and } \quad \sum\limits_{i=1}^n w_{in} m(O_i,\hat{\theta}_{n,BB},\hat{h}) =0.
    \]
  Our goal is to establish the asymptotic properties of 
\[
\sqrt{n}\left(\hat{\theta}_n-\theta_0\right)\quad \text{ and } \quad \sqrt{n}\left(\hat{\theta}_{n,BB}-\hat{\theta}_n\right)\mid O_1 \ldots O_n,
\]
and show that both limiting distributions are Normal with zero mean and variance $\Sigma$, where
\[
\Sigma = M_{\theta_0}^{-1} E_{P_O}\left\{m(O;\theta_0,h_0)m(O;\theta_0,h_0)^\top\right\} M_{\theta_0}^{-\top},
\]
and
 \[
   M_{\theta_0}=E_{P_0}\left\{\frac{\partial m(O;\theta_0,h_0)}{\partial \theta} \right\}
  \]
is an invertible $p \times p $ matrix. The estimator $\hat{h}$ could be an estimator of Bayesian nature, arising from non-parametric  methods such as BART \citep [see, for example][]{Hahn/etal:2020}. We underline that $\hat{\theta}_n$ is introduced exclusively as a tool that facilitates establishing the Bayesian/frequentist duality and constructing Bayesian credible intervals. When this duality does not hold, we emphasize the role played by $\hat{\theta}_n$ in sections \ref{NOCF} and \ref{Debias}. We make the following assumptions:
 \begin{assumption}{\label{nuisanceconsistency}}
 $h_0$ is consistently estimated by  $\hat{h} \in H_n$, where $\left\{H_i\right\}_{i=1}^\infty$ are  shrinking neighborhoods of $h_0$.
 \end{assumption}
 \begin{assumption}{\label{consistency}}
    $\hat{\theta}_{n,BB}$ and $\hat{\theta}_n$ are unconditionally consistent estimators of $\theta_0$.
\end{assumption}
Sufficient conditions so that Assumption \ref{consistency} holds are given in \citet{Newey:1994}. However, if the score $m$ is linear in $\theta$, Assumption \ref{consistency} is generally satisfied. 
 \begin{assumption}
    The score $m(O;\theta_0,h_0)$ is a Neyman orthogonal score, satisfying Definition \ref{NeymanOrth}. 
 \end{assumption}

\begin{assumption}{\label{GCA}}
    The class 
   \[
   \left\{\frac{\partial m\left(O;\theta,h \right)}{\partial \theta ^\top}, \Vert \theta-\theta_0 \Vert_{p,2}<\delta_1, \Vert h-h_0 \Vert_{\mathcal{H}} <\delta_2 \right\}
   \]
   is $P_O$-Glivenko-Cantelli for some $\delta_1,\delta_2>0$, and the function 
   \[
   (\theta,h) \to E_{P_O}\left\{\frac{\partial m\left(O;\theta,h \right)}{\partial \theta ^\top}\right\}
   \]
   is continuous.
\end{assumption}
\begin{assumption}{\label{DonA}}
    The class $\mathcal{G}=\left\{ m(O;\theta_0,h), \hspace{2mm} \Vert h -h_0 \Vert_{\mathcal{H}} <\delta \right\}$  is $P_O$-Donsker for some $\delta>0$, and 
    $E_{P_O}  \left\{\Vert m(O;\theta_0,h)-m(O;\theta_0,h_0) \Vert^2_{p,2}\right\} \to 0$ as $\Vert h -h_0 \Vert_{\mathcal{H}} \to 0$
\end{assumption}
Assumption \ref{DonA} is in general sufficient to ensure that the terms studied in Lemma \ref{Donsker} converge unconditionally to zero. While the Donsker assumptions may be limiting, recent results can ensure that both terms studied in Lemma \ref{Donsker} converge to $0$ as $n \to \infty$;  see \citet{Yiu/etal:2025}.  
 \begin{assumption}{\label{GateauxA}}
     $\sqrt{n} E_{P_{O}}\left\{m (O;\theta_0,\hat{h})\right\} \to 0 $ in probability as $n \to \infty$.
 \end{assumption}
Neyman orthogonality of the score $m$ is an important condition in ensuring that Assumption \ref{GateauxA} is satisfied. In section \ref{NOCF}, we study the effect of the absence of Neyman orthogonality on the marginal posterior of the targeted parameter. The following lemma, whose proof can be found in the Supplementary Material, gives sufficient conditions for Assumption \ref{GateauxA} to hold. 
 \begin{lemma}{\label{Gateaux}}
Let $f(t)=E_{P_O}\left[ m\left\{O;\theta_0,h_0+t(h-h_0)\right\}\right]$. If 
\[
\int_{0}^1 \Vert f''(t) \Vert_{p,2} dt=o(n^{-1/2})
\]uniformly in $h \in \mathcal{H}$, then  $\sqrt{n} E_{P_O}\left\{m(O;\theta_0,h)\right\} \to 0$ uniformly in $h \in \mathcal{H}$ as $n \to \infty$.
\end{lemma}
The second derivative $f''$ of $f$ depends on the rates at which the nuisance parameters need to be estimated. We require in general the reasonable estimation rate $\Vert \hat{h}-h_0 \Vert_{\mathcal{H}}=o_{P_{O}}(n^{-1/4})$. In many examples, the nuisance parameter $h$ consists of two distinct components $h_1$ and $h_2$ with respective true values $h_{10}$ and $h_{20}$, lying respectively in convex subsets $\mathcal{H}_1$ and $\mathcal{H}_2$ of two normed spaces $N_1$ and $N_2$. The nuisance set is hence $\mathcal{H}=\mathcal{H}_1 \times \mathcal{H}_2$, a convex subset of $N_1 \times N_2$. We can then define $\Vert h \Vert_{\mathcal{H}}=\Vert  (h_1,h_2)\Vert _{\mathcal{H}}:=\max(\Vert h_1 \Vert_{\mathcal{H}_1},\Vert h_2 \Vert_{\mathcal{H}_2}).$ Requiring that $h_{10}$ and $h_{20}$ are both estimated at an $o(n^{-1/4})$ rate suffices to ensure that $\Vert \hat{h}-h_0 \Vert_{\mathcal{H}}=o_{P_{O}}(n^{-1/4})$. There are more refined conditions that can be imposed on a case by case basis.  
\subsection{Asymptotic Properties}{\label{MainResults}}
We establish answers to the questions posed at the start of section \ref{DonskerAss} via a succession of lemmas. It is important to outline the derivation in order to highlight the role of each of the above assumptions.  We have that
\[
 0=\frac{1}{\sqrt{n}}\sum\limits_{i=1}^n m(O_i;\hat{\theta}_{n},\hat{h})=\frac{1}{\sqrt{n}}\sum\limits_{i=1}^nm(O_i;\theta_0,\hat{h})+\left\{ \frac{1}{n}\sum\limits_{i=1}^n \frac{\partial m(O_i;\overline{\theta}_n,\hat{h})}{\partial \theta ^\top} \right\}\sqrt{n}\left( \hat{\theta}_n -\theta_0 \right) 
\]
and 
\begin{align*}
0& =\sqrt{n}\sum\limits_{i=1}^n w_{in} m(O_i;\hat{\theta}_{n,BB},\hat{h})\\
& =\sqrt{n}\sum\limits_{i=1}^n w_{in} m\left(O_i;\theta_{0},\hat{h}\right) +\left\{\sum\limits_{i=1}^n w_{in}\frac{\partial m(O_i;\tilde{\theta}_n,\hat{h})}{\partial \theta ^\top}\right\} \sqrt{n}(\hat{\theta}_{n,BB}-\theta_0)
\end{align*}
for some $\overline{\theta}_n \in [\theta_0,\hat{\theta}_n]$ and $\tilde{\theta}_n \in [\theta_0,\hat{\theta}_{n,BB}]$, where $[a,b]$ denotes the line the segment between $a$ and $b$. Another expansion yields that
\begin{align*}
\sqrt{n}\sum\limits_{i=1}^n & w_{in} m(O_i;\theta_0,\hat{h}) \\
&=\sqrt{n}\sum\limits_{i=1}^n w_{in}m(O_i;\theta_0,h_0)
\ +\sqrt{n} \sum\limits_{i=1}^n \left(w_{in}-\frac{1}{n}\right)
\left\{m\left(O_i;\theta_0,\hat{h}\right)-m\left(O_i;\theta_0,h_0\right)\right\}  \\
&\ +\sqrt{n}\left[ \frac{1}{n}\sum\limits_{i=1}^n\left\{m(O_i;\theta_0,\hat{h})-m(O_i;\theta_0,h_0)\right\}-E_{P_O}\left\{m(O;\theta_0,\hat{h})-m(O;\theta_0,h_0)\right\}\right]\\[6pt]
&\ +\sqrt{n}E_{P_O}m(O;\theta_0,\hat{h}) \\
\end{align*}
We obtain, by setting the weights equal to $1/n$, an expansion of $n^{-1/2}\sum\limits_{i=1}^n m(O_i;\theta_0,\hat{h})$.

\begin{lemma}{\label{GC}} If Assumptions \ref{nuisanceconsistency},\ref{consistency}, and \ref{GCA} are satisfied,
    then both
   \[
    \frac{1}{n}\sum\limits_{i=1}^n \frac{\partial m(O_i;\overline{\theta}_n,\hat{h})}{\partial \theta ^\top}  \quad \hspace{2mm} \text{and } \hspace{2mm}\quad \sum\limits_{i=1}^n w_{in}\frac{\partial m(O_i;\tilde{\theta}_n,\hat{h})}{\partial \theta ^\top}
   \]
   converge unconditionally to $M_{\theta_0}$ in probability.
\end{lemma}
\begin{proof}
    This fact follows from the exchangeable bootstrap for Glivenko-Cantelli classes.
\end{proof}
\begin{lemma}{\label{Donsker}}
    If Assumptions \ref{nuisanceconsistency} and \ref{DonA} are satisfied, then 
   \begin{enumerate}
       \item
       \[
       \sqrt{n} \sum\limits_{i=1}^n \left(w_{in}-\frac{1}{n}\right)
\left\{m\left(O_i;\theta_0,\hat{h}\right)-m\left(O_i;\theta_0,h_0\right)\right\}
      \]
    \item 
    \[
    \sqrt{n} \left[\frac{1}{n}\sum\limits_{i=1}^n\left\{m(O_i;\theta_0,\hat{h})-m(O_i;\theta_0,h_0)\right\}-E_{P_O}\left\{m(O;\theta_0,\hat{h})-m(O;\theta_0,h_0)\right\}\right]
       \]
   \end{enumerate}
   converge unconditionally to zero in probability as $n \to \infty$.
\end{lemma}
\begin{proof}
    This fact follows from the exchangeable bootstrap for Donsker classes, in particular from the fact that the processes $\sqrt{n}(P_n^w-P_n)$ and $\sqrt{n}(P_n-P)$ converge unconditionally in distribution to a tight Gaussian process in $L^{\infty}(\mathcal{G})$.
\end{proof}

\begin{corollary}{\label{bootest}}
    If Assumptions \ref{nuisanceconsistency}-\ref{GateauxA} are satisfied, then:
    :
    \[
    \sqrt{n}(\hat{\theta}_{n,BB}-\theta_0)= -M_{\theta_0}^{-1} \sqrt{n}\sum\limits_{i=1}^n w_{in} m\left(O_i;\theta_0,h_0\right)+ o_{P_{OW}}(1)
    \]
\end{corollary}
We point out that the unconditional asymptotic distribution of $\sqrt{n}(\hat{\theta}_{n,BB}-\theta_0)$ is Normal with zero mean and variance $ 2 \Sigma$. Combining the previous lemmas, we obtain the following theorem:

\begin{theorem}{\label{maintheorem}}
    If Assumptions \ref{nuisanceconsistency}-\ref{GateauxA} are satisfied, then : 
    \begin{align*}
        \sqrt{n}(\hat{\theta}_{n}-\theta_0)&= -M_{\theta_0}^{-1} \frac{1}{\sqrt{n}}\sum\limits_{i=1}^n  m\left(O_i;\theta_0,h_0\right)+ o_{P_{O}}(1)\\
    \sqrt{n}(\hat{\theta}_{n,BB}-\hat{\theta}_n)&= -M_{\theta_0}^{-1} \sqrt{n}\sum\limits_{i=1}^n \left(w_{in}-\frac{1}{n}\right) m\left(O_i;\theta_0,h_0\right)+o_{P_{OW}}(1),   
\end{align*}
   Those results together imply that 
\[
\sqrt{n}(\hat{\theta}_n-\theta_0 ) \qquad \text{ and } \qquad
\sqrt{n} (\hat{\theta}_{n,BB}-\hat{\theta}_n )\mid O_1,\ldots ,O_n
\]
converge in distribution to random variables with $\mathcal{N}(0,\Sigma)$ distributions (almost surely).             
\end{theorem}
We establish this duality when the score $m$ is not necessarily differentiable with respect to $\theta$ in the Supplementary Material.\\

In light of Theorem \ref{maintheorem}, we adopt the following computational strategy to sample from a posterior distribution of a the targeted parameter $\theta_0$, in the presence of a nuisance parameter $h_0$, satisfying $E_{P_O}\left\{m(O;\theta_0,h_0)\right\}=0$ where $m$ is an orthogonal score. It is hence not necessary to propagate the inferential uncertainty attached to the estimation of the nuisance parameter
Indeed, an estimator of $h_0$ that is held fixed during the entire procedure ensures that the posterior distribution possesses optimal frequentist properties. Moreover, the conditional  distribution $\sqrt{n} (\hat{\theta}_{n,BB}-\hat{\theta}_n )\mid O_1,\ldots ,O_n$ can be seen as an approximation to the sampling frequentist density of $\sqrt{n}(\hat{\theta}_n-\theta_0).$
\begin{algo}{\label{AlgoPost}}
Sampling from the Posterior Distribution
\begin{tabbing}
   \qquad \enspace Estimate $\hat{h}$ (possibly non-parametrically) on $O_1,\ldots,O_n$.\\
   \qquad \enspace For $j=1$ to $j=B$ \\
   \qquad \qquad Draw random weights $(w^{(j)}_{1n},\ldots,w^{(j)}_{nn}) \sim \text{Dir}(1,\ldots,1)$\\
  \qquad \qquad Find the solution $\theta^{(j)}$ to
            $\sum_{i=1}^n w^{(j)}_{in}m(O_i;\theta,\hat{h})=0$\\
\qquad \enspace Output $\theta^{(1)},\ldots,\theta^{(B)}$
\end{tabbing}
\end{algo}

The distribution $P_W$ of the weights can be set to be $n^{-1} \text{Multinomial}(n,n^{-1},\ldots,n^{-1})$ without distorting the asymptotic results obtained in Theorem \ref{maintheorem}. Hence the process $P_n^w$ coincides with Efron's bootstrap. The estimate $\hat{h}$ can be obtained prior to bootstrapping and not updated with each bootstrap iteration. Also, different estimators $\hat{h}^{(1)},\ldots,\hat h^{(B)}$ may be used, with $\hat{h}^{(j)}$ used for the $j$-th bootstrap replicate without altering the result of Theorem \ref{maintheorem}, and this may yield better finite sample performance. As an example, suppose that $\hat{h}_1$ and $\hat{h}_2$ are estimators of $h_0$. For $j\in \left\{1,2\right\}$, we consider the solutions to
\[
\sum_{i=1}^n m(O_i,\theta,\hat{h}_j) =0.
\]
Applying Algorithm \ref{AlgoPost} successively for $\hat{h}_1$ and $\hat{h}_2$ and collating the results yields a posterior sample of size $B_1+B_2$

\section{Simulations}{\label{Sim}}
We consider the model in subsection  and implement Algorithm \ref{AlgoPost}. The data generating mechanism is as follows: $X \in R^q$ be distributed according to a multivariate Normal distribution with zero mean and covariance matrix $\Sigma_X \in R^{q \times q}$, whose $(i,j)$-th entry equals $0.8^{\vert i-j \vert}/4$. We then generate the binary treatment $Z$ and outcome $Y$ such that 
\begin{align}\label{gmod1}
    Y &= \theta_0 Z + g_0(X)  +U \\ 
    Z\mid X& \sim\text{Bern}\left\{e_0(X)\right\},
\end{align}
where $\theta_0=3$, $U \sim \mathcal{N}(0,1)$
\begin{align*}
    g_0(X)&=X_1+\sin(X_2+X_3)  + \cos(X_3)+\vert X_4 \vert +X_q\\
    e_0(X)&= \frac{1}{2} \frac{\exp\left(\sum\limits_{j=1}^{q} X_j\right)}{1+\exp\left(\sum\limits_{j=1}^{q} X_j\right)}+\frac{1}{5}
\end{align*}
In this section, we report the results for $q=5$ for three different sample sizes.

We estimate $k_y^{0}(X)=E_{P_O}\left(Y \mid X\right)$ and $e_0(X)$ using random-forests with sub-sampling without replacement with subsample size $m=n^{0.49}$. We refer to \citet{Chen/etal:2022} for a theoretical justification of the stochastic equicontinuity property satisfied by this procedure, as well as bounds on the mean-squared error rates achieved. Even though there are no theoretical guarantees for the convegence of the bootstrapped term, we implment Algorithm \ref{AlgoPost} for this model. The score used in this study is the partialled-out score
\[
m(O;\theta_0,h_0) = \left[Y-k_y^{0}(X)-\theta_0 \left\{Z-e_0(X)\right\}\right]\left\{Z-e_0(X)\right\},
\]
where $h_0=(k_y^0,e_0)$; see \citet{Robinson:1988}. We compute $f(t)=E_{P_O}\left[m\left\{O;\theta_0,h_0+t\left(h-h_0\right)\right\}\right]$ in detail in the Supplementary Material and show that
\[
    f(t) =t^2 {E}_{P_O}\left[\left\{k_y(X)-k_y^0(X)\right\}\left\{e(X)-e_0(X)\right\}-\psi_0\left\{e(X)-e_0(X)\right\}^2\right]
\]
and explain how it encapsulates the required rates of convergence of the nuisance parameter estimators to their respective true value. Similar calculations are found in \citet{Chernozhukov/etal:2018}, however our method is slightly different. \\

We run a simulation across $1000$ replicate analyses for different sample sizes $n$. For each replicate, we implement Algorithm \ref{AlgoPost}
 and derive the posterior distribution for the targeted parameter using 1000 Bayesian bootstrap samples. We report the results in Table \ref{ATEPLM}, which underscores the Bayesian/frequentist duality proved in Theorem \ref{maintheorem}. Additional numerical results where the nuisance parameters are fitted using conventional non-parametric methods are also found in the Supplementary Material. Additional calculations and simulations concerning the AIPW estimator of the ATE are found in the Supplementary Material. 
 \begin{table}
	\centering \caption{Characteristics of the frequentist and Bayesian distributions\break}
    {  \begin{tabular}{@{}lrrrr}
       \hline
       & \multicolumn{3}{c}{$n$} \\[3pt]
     &      \multicolumn{1}{c}{$250$}
& \multicolumn{1}{c}{$500$} & \multicolumn{1}{c}{$1000$}
\\
        \hline
		Average of the Posterior Means& $3.01$ & $3.02$ & $3.00$\\ 
		  Empirical Frequentist Mean& $3.01$ & $3.02$ & $3.00$ \\ 
		Average of Posterior Variances $(\times n)$& $5.36$ & $5.09$ & $4.98$ &\\ 
		Empirical Frequentist Variance $(\times n$) & $5.08$ & $4.71$ & $4.88$ \\ 
	    Average Sandwich Estimate   & $5.38$  & $5.11$  & $4.99$\\
       \hline
       Average Bayesian credible interval & $(2.73,3.30)$  &$(2.82,3.22)$ & $(2.87,3.15)$\\
       Frequentist confidence interval &$(2.74,3.30)$ &$(2.83,3.21)$ &$(2.87,3.15)$\\
       Posterior Coverage    & $94.80$  & $95.30$  & $94.30$ 
    \end{tabular}}
	\label{ATEPLM}
	\label{tab2}
\end{table}
\section{Discussion}
\subsection{Neyman Orthogonality and Cutting Feedback}{\label{NOCF}}
The Neyman orthogonality condition has been interpreted in the frequentist literature as insensitivity to biased (but consistent) estimation of the nuisance parameter $h_0$. In light of the results presented above, we give Neyman orthogonality a Bayesian interpretation as a guarantee of the robustness of the posterior to biased but consistent estimation of $h_0$ and the insensitivity of the posterior to uncertainty propagation. That is, the posterior computed through the Bayesian bootstrap satisfies a Bayesian/frequentist duality statement despite the nuisance parameter being estimated only once.  Cutting feedback using a conventional plug-in approach is hence a valid way to perform Bayesian inference, provided a Neyman orthogonal score is used.  It is important to reiterate that, in most parametric models, even in purely frequentist settings, estimating a nuisance parameter will affect the variance of the parameter of interest. Many scores used in parametric statistics are not orthogonal with respect to the nuisance parameter. In the parametric case, estimating the posterior distribution of parameter of interest will generally be affected by the use of a plug-in estimate of  the nuisance parameter. The asymptotic Normality of this posterior, however, is always guaranteed even if the posterior has a higher than desired variance and exhibits inadequate coverage \citep{Sabbagh/Stephens:2026}.

\textcolor{black}{When the nuisance parameter is estimated non-parametrically, asymptotic Normality of the posterior may not be secured.} Nevertheless, imposing Neyman orthogonality of the score function coupled with reasonable requirements on the nuisance estimator will lead to an asymptotically Normal posterior distribution unaffected by the estimation of $h_0$. If the conclusion of Lemma \ref{Donsker} holds, but Neyman orthogonality is not satisfied, and if instead $\sqrt{n} E_{P_O}\left\{m(O;\theta_0,\hat{h})\right\}$  is assumed to be bounded in probability, we obtain that 
\begin{align*}
\sqrt{n}(\hat{\theta}_{n}-\theta_0)&= -M_{\theta_0}^{-1} \frac{1}{\sqrt{n}}\sum\limits_{i=1}^n  m\left(O_i;\theta_0,h_0\right)-M_{\theta_0}^{-1} \sqrt{n} E_{P_O}\left\{m(O;\theta_0,\hat{h})\right\}+o_{P_{O}}(1),\\
 \sqrt{n}(\hat{\theta}_{n,BB}-\theta_0)&= -M_{\theta_0}^{-1} \sqrt{n}\sum\limits_{i=1}^n w_{in} m\left(O_i;\theta_0,h_0\right)-M_{\theta_0}^{-1} \sqrt{n} E_{P_O}\left\{m(O;\theta_0,\hat{h})\right\}+o_{P_{OW}}(1).
\end{align*}
Therefore,
\[
    \sqrt{n}(\hat{\theta}_{n,BB}-\hat{\theta}_n)= -M_{\theta_0}^{-1} \sqrt{n}\sum\limits_{i=1}^n \left(w_{in}-\frac{1}{n}\right) m\left(O_i;\theta_0,h_0\right)+o_{P_{OW}}(1),   
\]
which implies that, almost surely, as $n \to \infty$,
\[
\sqrt{n}(\hat{\theta}_{n,BB}-\hat{\theta}_n) \mid O_1,\ldots,O_n \to \mathcal{N}(0,\Sigma).
\]
While this result is appealing at first, it transpires that the posterior may manifest inadequate coverage and inflated variance. The Bayesian/frequentist duality may not hold. In fact, if $\sqrt{n} E_{P_O}\left\{m(O;\theta_0,\hat{h})\right\}$ is regular and asymptotically linear with a non-zero influence function \citep{Newey:1994}, the asymptotic frequentist variance of $\sqrt{n}\left(\hat{\theta}_n-\theta_0\right)$ may be no larger than $\Sigma$. 

\subsection{Debiasing and Convergence of the Posterior}{\label{Debias}}

The frequentist estimator $\hat{\theta}_n$, the solution of $\sum_{i=1}^n m(O_i;\theta,\hat{h})=0$, provides the correct way of debiasing the posterior regardless of whether the score is orthogonal. Such calculations were essential in establishing results of similar flavor for two-step approaches in \citet{Sabbagh/Stephens:2026} whereby the asymptotic Normality of the posterior in the parametric setting is established, albeit with coverage and the Bayesian/frequentist duality are compromised. The following example underlines that the posterior under the Bayesian bootstrap converges under rather minor restrictions. This convergence is revealed provided the debiasing is carried out correctly. We suppose that the score $m$ identifying a real-valued parameter $\theta$ has the form,  
\[
m(O;\theta,h)=\theta-B(O;h),
\]
for some function $B: \mathcal{O} \times \mathcal{H} \to {R}$, where $E_{P_O}\left\{m(O;\theta_0,h_0)\right\}=0$. Let $\hat{h}$ be an estimator of $h$. The respective  solutions to 
\[
\sum_{i=1}^n  m(O_i;\theta,\hat{h})=0 \quad \text{ and } \quad  \sum_{i=1}^n w_{in}  m(O_i;\theta,\hat{h})=0 
\]
are
\[
\hat{\theta}_n=\frac{1}{n}\sum_{i=1}^n B(O_i,\hat{h}) \quad  \text{ and } \quad\hat{\theta}_{n,BB}=\sum_{i=1}^n w_{in}B(O_i,\hat{h}).
\]
Then the following theorem holds:
\begin{theorem}{\label{PostConv}}
    If Assumption \ref{DonA} or any of its possible alternatives holds, then almost surely, as $n \to \infty$
    \[
    \sqrt{n} \left(\hat{\theta}_{n,BB}-\hat{\theta}_n\right) \mid O_1,\ldots O_n \to \mathcal{N}\left(0, E_{P_O} \left\{m(O;\theta_0,h_0)^2\right\}\right).
    \]
\end{theorem}

\begin{proof}
We write
\begin{align*}
\sqrt{n} \left(\hat{\theta}_{n,BB}-\hat{\theta}_n\right)&=\sqrt{n}\sum_{i=1}^n \left(w_{in}-\frac{1}{n} \right) B(O_i,\hat{h})\\
&=\sqrt{n}\sum_{i=1}^n \left(w_{in}-\frac{1}{n}\right) \left\{B(O_i,\hat{h})-B(O_i,h_0)\right\}+\sqrt{n}\sum_{i=1}^n \left(w_{in}-\frac{1}{n} \right)B(O_i,h_0).
\end{align*}
The first term in the right hand side can be assumed to go to zero, under Donsker assumptions, say. Such results typically merely require $\Vert \hat{h}-h_0 \Vert_{\mathcal{H}}=o_{P_O}(1)$, rather than the faster rate of $o_{P_O}(n^{-1/4})$. We then obtain that, almost surely,
\[
\sqrt{n} \left(\hat{\theta}_{n,BB}-\hat{\theta}_n\right) \mid O_1,\ldots O_n \to \mathcal{N}\left(0, E_{P_O} \left\{m(O;\theta_0,h_0)^2\right\}\right),
\]
irrespective of whether the score $m$ is Neyman orthogonal or not. 
\end{proof}

The conditional mean of $\hat{\theta}_{n,BB}$, which is  seen to $\hat{\theta}_n$, provides the most natural way to debias $\hat{\theta}_{n,BB}$. It may be tempting to debias $\hat{\theta}_{n,BB}$ by quantities that do not depend on $\hat{h}$ such as $n^{-1} \sum_{i=1}^n B(O_i,h_0)$. Such debiasing is not optimal, as it is not based on the conditional mean and does not allow recovery of the results of Theorem \ref{PostConv} without imposing additional unnecessary assumptions. The asymptotic posterior variance corresponds to the asymptotic variance of the frequentist and Bayesian estimators based on the equations $\sum_{i=1}^n m(O_i,\theta,h_0)=0$ and $\sum_{i=1}^n w_{in} m(O_i,\theta,h_0)=0$. These fundamental calculations reveal that under weak restrictions, posterior distributions are asymptotically not influenced by the nuisance parameter, irrespective of the score used. Recovering the Bayesian/frequentist duality and coverage at the nominal level depends on  $\hat{\theta}_n$, the frequentist estimator, and its properties dictated by the Neyman orthogonality of the score $m$ and the fast enough convergence rate of $\hat{h}$ to $h_0$. This highlights the fact that ensuring an asymptotically normal posterior with the desired variance is not enough to ensure valid confidence intervals. The existence of a 'bootstrap consistency' statement leading to 'a validity of the bootstrap' conclusion, and the pairing of a conditional quantity to a corresponding unconditional quantity, are required; see \citet{Cheng/Huang:2010}.

\subsection{Future Work}
The Donsker assumptions attached to the nuisance space played a crucial role in establishing the properties of the frequentist and Bayesian estimators of the targeted parameter of interest. Indeed, the properties of the exchangeable bootstrap for Donsker classes enables us to swiftly prove Lemma \ref{Donsker}. With the rapid expansion of machine learning algorithms, interest has expanded beyond reliance on Donsker assumptions. It is well documented  in \citet{Chernozhukov/etal:2018} that Donsker classes of functions may fail to correctly model nuisance functions that depend on a large number of covariates. Approaches that allow relaxation of the Donsker assumptions in order to incorporate a larger class of flexible methods as estimators for the nuisance parameters have been proposed in \citet{Sabbagh2025}.  These approaches deploy cross-fitting procedures that allow the Bayesian/frequentist duality to be maintained even if Donsker conditions do not apply, such as when the dimensionality of the predictor space is high.

\section*{Acknowledgement}
The authors acknowledge the support of the Natural Sciences and Engineering Research Council of Canada (NSERC) and the  Institut des Sciences Math\'ematiques (ISM).

\bigskip

\newtheorem{proposition}{Proposition}

\begin{center}
{\large \textsc{Supplementary Material}}
\end{center}

The Supplementary Material includes the assumptions on the relevant probability spaces, a proof of the main theorem without the differentiability of the score, calculations on the partially linear and partial interactive models and additional simulations concerning the AIPW estimator.

\bigskip

\setcounter{section}{0}

\makeatletter
\renewcommand \thesection{S.\@arabic\c@section}
\renewcommand\thetable{S.\@arabic\c@table}
\renewcommand \thefigure{S.\@arabic\c@figure}
\makeatother

\section{Probability Spaces}
The following assumptions specify the relevant probability spaces required in the theoretical statements. We assume
\begin{enumerate}
    \item The data are realizations from the from probability space $(\mathcal{O},B,P_O)$.
    \item The Bayesian bootstrap weights are defined on a probability space $(\mathcal{W},\mathcal{C},P_W)$.
    \item Observation $O_i$ is the $i$-th coordinate of the canonical projection from $(\mathcal{O}^\infty,B^\infty,P_O ^\infty)$.
    \item For the joint randomness, coming from the observed data and from the weights, is defined on the product probability space
    \[
    (\mathcal{O}^\infty,B^\infty,P_O ^\infty) \times (\mathcal{W},\mathcal{C},P_W)=(\mathcal{O}^\infty \times \mathcal{W},B^\infty \times \mathcal{C},P_O ^\infty \times P_W).
    \]
    The joint probability measure is hence $P_{OW}=P_{O}^\infty \times P_W$. We write $P_O$ in lieu of $P_O^{\infty}$ for simplicity when necessary. 
\end{enumerate}
\section{Proof of Lemma \ref{Gateaux}}{\label{GateauxProof}}
\begin{proof}[Proof of Lemma \ref{Gateaux}]
    The proof consists of a second order expansion of $f$ around $0$ and an application of a form of a mean value theorem.  By Lagrange's mean value theorem as in \citet{Bartle/Sherbert:2020}, we have that 
    $$
    f(1) = f(0) +f'(0) + \frac{1}{2} \int\limits_0^1 (1-s)f''(s) ds,
    $$
    Note that $f(1)=E_{P_O}\left\{m\left(O;\theta_0,h\right)\right\}$ and $f(0)=E_{P_O}\left\{m(O;\theta_0,h_0)\right\}=0$ and $f'(0)=0$ by Neyman orthogonality.
    Moreover, 
    \[
    \left\Vert\int\limits_0^1 (1-x)f''(x) dx \right\Vert_{p,2} \le \int\limits_0^1 \left( 1-x\right)  \Vert f''(x) \Vert_{p,2} \ dx\le \int \limits_0^1 \Vert f''(x)  \Vert_{p,2} \ dx
    \]
    Therefore, 
    \[
    \left\Vert \sqrt{n} f(1)\right\Vert_{p,2} \leq \sqrt{n} \int\limits_0^1 \left\Vert f''(x) \right\Vert_{p,2} \ dx  \to 0.
    \]
    This shows that uniformly in $h$, we have
    \[
    \sqrt{n} E_{P_O}\left\{m(O;\theta_0,h)\right\}\to 0,
    \]
    which in turns leads to the fact
    \[
    \sqrt{n}E_{P_O}\left\{m\left(O;\theta_0,\hat{h}\right)\right\} \text{ converges to 0 in probability as } n \to \infty.
    \]
\end{proof}
\section{Extension to non-differentiable scores}
We use the same framework as the one in section \ref{MainResults} of the main paper. We suppose that $E_{P_O}\left\{m(O;\theta_0,h_0)\right\}=0$, where $m$ is a Neyman orthogonal score with respect to $h$.
Let $\hat{h}$ be an estimator of $h_0$ and denote by $\hat{\theta}_n$ and $\hat{\theta}_{n,BB}$ the respective solutions to 
 \[
     \sum\limits_{i=1}^n m(O_i,\hat{\theta}_{n},\hat{h})=0 \quad \text{ and } \quad \sum\limits_{i=1}^n w_{in} m(O_i,\hat{\theta}_{n,BB},\hat{h}) =0.
    \]
Instead of assuming that $m$ is differentiable with respect to $\theta$, we assume that the function  $(\theta, h) \mapsto E_{P_{O}}\left[m \left\{O;\theta,h\right\}\right]$ is differentiable with respect to $\theta$ and $h$. Concretly, we assume that 
the function \[
g(s)=E_{P_{O}}\left[m \left\{O;\theta_0+s(\theta-\theta_0),h_0+s(h-h_0)\right\}\right]
\]
is differentiable at $0$ for every $h \in \mathcal{H}$ and $\theta \in \Theta$. The targeted parameter space is a convex subset of $R^p$, and $\theta_0$ is in the interior of $\Theta$. Similar constraints can be placed on the nuisance space $\mathcal{H}$. The main idea is to be able to define the derivative of $g(s)$ at $s=0$. We remark that $g(0)=E_{P_O}\left\{m(O;\theta_0,h_0)\right\}=0$ by assumption and that $g'(0)= M_{\theta_0}(\theta-\theta_0)$, by Neyman orthogonality of $m$ with respect to $h$, where 
\[
M_{\theta_0}=  \left. \frac{\partial\left[{E}_{P_O} \left\{m\left(O;\theta,h_0 \right)\right\}\right]}{\partial {\theta}^\top} \right |_{\theta = \theta_0}.
\]
By also assuming that $M_{\theta_0}$ is an invertible, and by defining 
\[
\Sigma = M_{\theta_0}^{-1} E_{P_O}\left\{m(O;\theta_0,h_0)m(O;\theta_0,h_0)^\top\right\} M_{\theta_0}^{-\top}
\]
we show that 
\[
\sqrt{n}(\hat{\theta}_n-\theta_0 ) \qquad \text{ and } \qquad
\sqrt{n} (\hat{\theta}_{n,BB}-\hat{\theta}_n )\mid O_1,\ldots ,O_n
\]
converge in distribution to random variables with $\mathcal{N}(0,\Sigma)$ distributions (almost surely).   

We make the following assumptions :
\begin{assumption}{\label{nuisanceconsistencyN}}
 $h_0$ is consistently estimated by  $\hat{h} \in \mathcal{H}_n$, where $\left\{\mathcal{H}_i\right\}_{i=1}^\infty$ are  shrinking neighborhoods of $h_0$.
 \end{assumption}
 \begin{assumption}{\label{consistencyN}}
    $\hat{\theta}_{n,BB}$ and $\hat{\theta}_n$ are unconditionally consistent estimators of $\theta_0$.
\end{assumption}

\begin{assumption}{\label{DonAN}}
    The class $\mathcal{G}=\left\{ m(O;\theta,h), \hspace{2mm} \Vert \theta -\theta_0 \Vert_{p,2} < \delta_1,\Vert h -h_0 \Vert_{\mathcal{H}} <\delta_2 \right\}$  is $P_O$-Donsker for some $\delta_1,\delta_2>0$, and 
    $E_{P_O}  \left\{\Vert m(O;\theta,h)-m(O;\theta_0,h_0) \Vert^2_{p,2}\right\} \to 0$ as $\Vert h -h_0 \Vert_{\mathcal{H}} \to 0$ and $\Vert \theta-\theta_0\Vert_{P,2} \to 0$.
\end{assumption}
 \begin{assumption}{\label{GateauxAN}}
     The function $g(s)=E_{P_{O}}\left[m \left\{O;\theta_0+s(\theta-\theta_0),h_0+s(h-h_0)\right\}\right]$ satisfies
     \[
     \sqrt{n}\int_0^1 \Vert g''(s)\Vert_{p,2} \ ds \to 0
     \]
     as $n \to \infty$ uniformly in $\theta$ and $h$.
 \end{assumption}
 Hence if Assumption \ref{GateauxAN} holds, and proceeding as in the proof of Lemma \ref{Gateaux}, we obtain that
\[
\sqrt{n}E_{P_{O}}\left\{m \left(O;\theta,h\right)\right\}-M_{\theta_0}\sqrt{n}(\theta-\theta_0)
\]
converges to $0$ as $n \to \infty$ uniformly in $\theta$ and $h$. By considering the expansion, and using the properties of the exchangeable bootstrap for Donsker classes, we obtain
\begin{align*}
0
0&=\sqrt{n}\sum\limits_{i=1}^n  w_{in} m(O_i;\hat{\theta}_{n,BB},\hat{h}) \\
&=\sqrt{n}\sum\limits_{i=1}^n w_{in}m(O_i;\theta_0,h_0)
\ +\sqrt{n} \sum\limits_{i=1}^n \left(w_{in}-\frac{1}{n}\right)
\left\{m\left(O_i;\hat{\theta}_{n,BB},\hat{h}\right)-m\left(O_i;\theta_0,h_0\right)\right\}  \\
&\ +\sqrt{n}\left[ \frac{1}{n}\sum\limits_{i=1}^n\left\{m(O_i;\hat{\theta}_{n,BB},\hat{h})-m(O_i;\theta_0,h_0)\right\}-E_{P_O}\left\{m(O;\hat{\theta}_{n,BB},\hat{h})-m(O;\theta_0,h_0)\right\}\right]\\[6pt]
&\ +\sqrt{n}E_{P_O}m(O;\hat{\theta}_{n,BB},\hat{h}) \\
&=\sqrt{n}\sum\limits_{i=1}^n w_{in}m(O_i;\theta_0,h_0)+\sqrt{n}E_{P_O}m(O;\hat{\theta}_{n,BB},\hat{h})+o_{P_{OW}}(1)\\
&=\sqrt{n}\sum\limits_{i=1}^n w_{in}m(O_i;\theta_0,h_0)+M_{\theta_0} \sqrt{n}\left(\hat{\theta}_{n,BB}-\theta_0\right)+o_{P_{OW}}(1)
\end{align*}
Therefore,
\[
\sqrt{n}\left(\hat{\theta}_{n,BB}-\theta_0\right)=-M_{\theta_0}^{-1}\sqrt{n}\sum\limits_{i=1}^n w_{in}m(O_i;\theta_0,h_0)+o_{P_{OW}}(1).
\]
In a similar manner, we obtain that
\[
  \sqrt{n}(\hat{\theta}_{n}-\theta_0)= -M_{\theta_0}^{-1} \frac{1}{\sqrt{n}}\sum\limits_{i=1}^n  m\left(O_i;\theta_0,h_0\right)+ o_{P_{O}}(1)
  \]
which results in 
\[
\sqrt{n}(\hat{\theta}_{n,BB}-\hat{\theta}_n)= -M_{\theta_0}^{-1} \sqrt{n}\sum\limits_{i=1}^n \left(w_{in}-\frac{1}{n}\right) m\left(O_i;\theta_0,h_0\right)+o_{P_{OW}}(1),   
\]
establishing the desired Bayesian/frequentist duality.
\section{Additional Calculations}{\label{GateauxCalculations}}
In what follows, we furnish some examples of models and various parameters of interest that satisfy an orthogonal score and give practical insights on how to efficiently check the assumptions related to $f(t)=E_{P_O}\left[ m\left\{O;\theta_0,h_0+t(h-h_0)\right\}\right]$. \citet{Chernozhukov/etal:2018} is extensive in its examples showing these fundamental calculations, and some calculations showed below have been done therein. However, there are some tips we provide that one can take into account that may render these calculations slightly simpler and more manageable.\\

\subsection{Partially linear model}{\label{secPartiallyLinearModel}}
 Consider the structural model 
 \begin{align*}
 Y& = \theta_0 Z + g_0(X) + U, \quad E_{P_O}(U \mid X,Z)=0 \\
Z \mid X  &\sim\text{Bern} \left\{e_0(X)\right\}.
\end{align*}

We first begin by looking at the partialled-out score of \citet{Robinson:1988}.

\[
m(O,\theta_0,h_0)= \left[Y-k^0_{y}(X)-\theta_0 \left\{Z-e_0(X)\right\}\right]\left\{Z-e_0(X)\right\},
\]
where $k^{0}_y(x)=E_{P_O}(Y|X=x)=\theta_0 e_0(x) + g_0(x)$.

A useful approach while computing $f(t)$ is to try to get a simplified expression to the largest possible extent before we embark on differentiating. In other words, we can try to simplify $f(t)$ and keep only nuisance parameters (although it may not be necessary in some cases) in its final expression. We illustrate with the following with the relatively simple example of the partialled-out score. The simplification provided by this method is amplified in the harder example of subsection \ref{PartialInteractiveModel}. In this case, the true parameter $h_0=(k_y^0,e_0)$. 
\begin{proposition}{\label{GateauxDerivativePartialLinearModelPartialledOutScore}}
    Let $f(t)= E_{P_O}\left[m\left\{O;\theta_0,h_0+t\left(h-h_0\right)\right\}\right]$, where $h=(k_y,e)$.
    Then, 
    \[f(t) =t^2 E_{P_O}\left[\left\{k_y(X)-k_y^0(X)\right\}\left\{e(X)-e_0(X)\right\}-\theta_0\left\{e(X)-e_0(X)\right\}^2\right] \]
\end{proposition}
In order to establish this equality, we use the facts that $E_{P_O}[V|X]=0$, where $V=Z-e_0(X)$ and $E_{P_O}[U|X,Z]=0$. One can see that
\[
f(t)=E_{P_O} \left[U-t\left\{k_y(X)-k_y^{0}(X)\right\}+\theta_0 t\left\{e(X)-e_0(X)\right\}\right]\left[V-t\left\{e(X)-e_0(X)\right\}\right],
\]
which simplifies to
\[
f(t) =t^2 E_{P_O}\left[\left\{k_y(X)-k_y^0(X)\right\}\left\{e(X)-e_0(X)\right\}-\theta_0\left\{e(X)-e_0(X)\right\}^2\right].
\] 
\noindent It becomes quite clear that $f(0)=0$ and $f'(0)=0$ and $f''(t)=f''(0)$,
and hence 
\begin{align*}
\int_{0}^1 \vert f''(t) \vert dt 
&\leq  2\sqrt{E_{P_O}\left[\left\{k_y(X)-k_y^0(X)\right\}^2\right]} \sqrt{E_{P_O}\left[\left\{e(X)-e_0(X)\right\}^2\right]} +2\theta_0E_{P_O}\left\{e(X)-e_0(X)\right\}^2,
\end{align*}
which hence implies that 
$$
\sqrt{n} \int_{0}^1 \vert f''(t) \vert dt \to 0, 
$$
if the nuisance parameters are estimated at a root-mean-squared error rate $o(n^{-1/4}).$
\subsection{Partial interactive model}{\label{PartialInteractiveModel}}
We now consider a generalized version of the partially linear model studied in \citet{Robinson:1988}. In this model, interactions between $X$ (confounders) and $Z$ (treatment) are allowed, and the outcome model is not assumed to be separable into a function of $X$ and a scalar multiple of $Z$. Suppose the structural model is 
\begin{align*}
Y &= \mu_0(Z,X) + U, \quad E_{P_O}[U|X,Z]=0\\
Z \mid X  &\sim\text{Bern} \left\{e_0(X)\right\},
\end{align*}
The following assumption, called positivity or no-overlap is prevalent in causal inference. It consists of bounding (almost surely) $\epsilon< P(Z=1|X) <1-\epsilon$ for some $0<\epsilon<1/2$. In practice, it means that treatment assignment is not deterministic. In theory, this property permits the bounded behavior of terms of the form $1/e(X)$ and $1/\left\{1-e(X)\right\}$. We will assume that $e_0$ satisfies the positivity property and the elements of the nuisance space satisfy it as well.  The average treatment effect $\theta_0$ in this model is equal to : 
\[
\theta_0 = E_{P_O}\left[\mu_0(1,X)-\mu_0(0,X) + \frac{Z\left\{Y-\mu_0(1,X)\right\}}{e_0(X)}- \frac{(1-Z)\left\{Y-\mu_0(0,X)\right\}}{1-e_0(X)}\right].
\]
In fact, this moment restriction leads to the well-known AIPW (Augmented Inverse Probability Weighting) estimator introduced in \citet{Robins/etal:1995}. 
We compute $f(t)=E_{P_O}\left[ m\left\{O;\theta_0,h_0+t\left(h-h_0\right)\right\}\right]$. It is better in such calculations to simplify $f(t)$ as much as possible before differentiating.
\begin{proposition}{\label{PropAIPWGateaux}}
    Let $f(t)=E_{P_O}\left[m\left\{O;\theta_0,h_0+t\left(h-h_0\right)\right\}\right]$, where 
    \[
    m(O;\theta_0,h_0)=- \theta_0 +\mu_0(1,X)-\mu_0(0,X) + \frac{Z\left\{Y-\mu_0(1,X)\right\}}{e_0(X)}- \frac{(1-Z)\left\{Y-\mu_0(0,X)\right\}}{1-e_0(X)},
    \]
      \begin{enumerate}
        \item 
        \begin{align*}
        f(t)=& E_{P_O} \left[\mu_0(1,X)+t\left\{\mu(1,X)-\mu_0(1,X)\right\}\right] 
    - E_{P_O} \left[\mu_0(0,X)+t\left\{\mu(0,X)-\mu_0(0,X)\right\}\right]\\
    &+ E_{P_O} \left[\frac{-te_0(X)\left\{\mu(1,X)-\mu_0(1,X)\right\}}{e_0(X)+t\left\{e(X)-e_0(X)\right\}}\right]-E_{P_O}\left[\frac{\left\{1-e_0(X)\right\}\left\{\mu(0,X)-\mu_0(0,X)\right\}t}{1-e_0(X)-t\left\{e(X)-e_0(X)\right\}}\right]-\theta_0
     \end{align*}
     \item 
     \begin{align*}
    f'(t)=& E_{P_O}\left\{ \mu(1,X)-\mu_0(1,X)\right\} 
    - E_{P_O} \left\{\mu(0,X)-\mu_0(0,X)\right\}\\
    &+ E_{P_O} \left[\frac{-e_0(X)^2\left\{\mu(1,X)-\mu_0(1,X)\right\}}{\left[e_0(X)+t\left\{e(X)-e_0(X)\right\}\right]^2}\right]
     -E_{P_O} \left[\frac{\left\{1-e_0(X)\right\}^2\left\{\mu(0,X)-\mu_0(0,X)\right\}}{\left[1-e_0(X)-t\left\{e(X)-e_0(X)\right\}\right]^2}\right],
\end{align*}
\item 
\begin{align*}
f''(t)&=2E_{P_O} \left[\frac{e_0(X)^2\left\{\mu(1,X)-\mu_0(1,X)\right\}\left\{e(X)-e_0(X)\right\}}{\left[e_0(X)+t\left\{e(X)-e_0(X)\right\}\right]^3}\right]\\
&-2E_{P_O} \left[\frac{\left\{1-e_0(X)\right\}^2\left\{\mu(0,X)-\mu_0(0,X)\right\}\left\{e(X)-e_0(X)\right\}}{\left[1-e_0(X)-t\left\{e(X)-e_0(X)\right\}\right]^3}\right].
\end{align*}
    \item In particular, $f(0)=0$ and $f'(0)=0$ and 
    $$
\sqrt{n} \int_{0}^{1}  \vert f''(t) \vert dt \to 0 \text{ as } n \to \infty,
$$
provided the nuisance parameters converge at an expected mean-squared error rate of order $o(n^{-1/4})$, in probability.
    \end{enumerate}
\end{proposition}
\begin{proof}
    \begin{align*}
f(t)=& E_{P_O} \left[\mu_0\left(1,X\right)+t\left\{\mu\left(1,X\right)-\mu_0\left(1,X\right)\right\}\right] 
    - E_{P_O}\left[ \mu_0(0,X)+t\left\{\mu(0,X)-\mu_0(0,X)\right\}\right]\\
    &+ E_{P_O} \left[\frac{Z\left[Y-\mu_0(1,X)-t\left\{\mu(1,X)-\mu_0(1,X)\right\}\right]}{e_0(X)+t\left\{e(X)-e_0(X)\right\}}\right]\\
   & - E_{P_O} \left[\frac{(1-Z)\left[Y-\mu_0(0,X)-t\left\{\mu(0,X)-\mu_0(0,X)\right\}\right]}{1-e_0(X)-t\left\{e(X)-e_0(X)\right\}}\right]-\theta_0,
\end{align*}
Instead of differentiating immediately, we have the ability to simplify $f$ by using the properties  $E_{P_O}[U|X,Z]=0$ and $E_{P_O}[Z|X]=e_0(X)$ and by noting that $Y=\mu_0(Z,X)+U$.
We obtain the following expression for $f(t)$.
\begin{align*}
f(t)=& E_{P_O} \left[\mu_0(1,X)+t\left\{\mu(1,X)-\mu_0(1,X)\right\}\right] 
    - E_{P_O} \left[\mu_0(0,X)+t\left\{\mu(0,X)-\mu_0(0,X)\right\}\right]\\
    &+ E_{P_O} \left[\frac{-te_0(X)\left\{\mu(1,X)-\mu_0(1,X)\right\}}{e_0(X)+t\left\{e(X)-e_0(X)\right\}}\right]
  -E_{P_O}\left[ \frac{\left\{1-e_0(X)\right\}\left\{\mu(0,X)-\mu_0(0,X)\right\}t}{1-e_0(X)-t\left\{e(X)-e_0(X)\right\}}\right]-\theta_0.
\end{align*}
We now differentiate $f$ term by term to get :
\begin{align*}
    f'(t)=& E_{P_O}\left\{ \mu(1,X)-\mu_0(1,X)\right\} 
    - E_{P_O} \left\{\mu(0,X)-\mu_0(0,X)\right\}\\
    &+ E_{P_O} \left[\frac{-e_0(X)^2\left\{\mu(1,X)-\mu_0(1,X)\right\}}{\left[e_0(X)+t\left\{e(X)-e_0(X)\right\}\right]^2}\right\}
     -E_{P_O} \left[\frac{\left\{1-e_0(X)\right\}^2\left\{\mu(0,X)-\mu_0(0,X)\right\}}{\left[1-e_0(X)-t\left\{e(X)-e_0(X)\right\}\right]^2}\right].
\end{align*}
Therefore, 
\begin{align*}
f'(0)& =E_{P_O} \left\{\mu(1,X)-\mu_0(1,X)\right\}-E_{P_O} \left\{\mu(0,X)-\mu_0(0,X)\right\} \\
& \qquad -E_{P_O} \left\{\mu(1,X)-\mu_0(1,X)\right\}+E_{P_O} \left\{\mu(0,X)-\mu_0(0,X)\right\}=0.
\end{align*}
Now, we compute $f''(t)$.
\begin{align*}
f''(t)&=2E_{P_O} \left[\frac{e_0(X)^2\left\{\mu(1,X)-\mu_0(1,X)\right\}\left\{e(X)-e_0(X)\right\}}{\left[e_0(X)+t\left\{e(X)-e_0(X)\right\}\right]^3}\right]\\
&-2E_{P_O} \left[\frac{\left\{1-e_0(X)\right\}^2\left\{\mu(0,X)-\mu_0(0,X)\right\}\left\{e(X)-e_0(X)\right\}}{\left[1-e_0(X)-t\left\{e(X)-e_0(X)\right\}\right]^3}\right].
\end{align*}
By keeping in mind that denominators are positive due to the no-overlap property, and after a simple integration of the absolute value of each of the terms on $[0,1]$, we can bound the integral of the absolute value of the first term by
\begin{align*}
 &E_{P_O}\left\{ \frac{\vert \mu(1,X)-\mu_0(1,X)\vert \left\vert e(X)-e_0(X) \right\vert e(X)+e_0(X)\vert }{e(X)^2}\right\} \\
 &\leq \frac{2-2 \epsilon }{\epsilon^2} 
 \sqrt{E_{P_O}\left[\left\{\mu(1,X)-\mu_0(1,X)\right\}^2\right]}\sqrt{E_{P_O}\left[\left\{e(X)-e_0(X)\right\}^2\right]},
\end{align*}
and the integral on $[0,1]$ of the absolute value of the second term by
\begin{align*}
&E_{P_O} \left\{\frac{\vert \mu(0,X)-\mu_0(0,X)\vert \left\vert e(X)-e_0(X)\right\vert  \vert 2-e(X)-e_0(X)\vert}{\left\{1-e(X)\right\}^2}\right\} \\
 &\leq \frac{2-2 \epsilon }{\epsilon^2} \sqrt{E_{P_O}\left[\left\{\mu(0,X)-\mu_0(0,X)\right\}^2\right]}\sqrt{E_{P_O}\left[\left\{e(X)-e_0(X)\right\}^2\right]},
\end{align*}
which hence implies that 
\[
\sqrt{n} \int_{0}^{1}  \vert f''(t) \vert dt \to 0 \text{ as } n \to \infty,
\]
provided the nuisance parameters converge at a root mean-squared error rate of order $o(n^{-1/4})$, in probability.
\end{proof}
\section{Additional Simulations}
\subsection{Simulations for the partially linear model using kernel methods}
We consider the following model, where the covariate $X \sim \mathcal{N}(0,1).$
\begin{align*}
Y&=\theta_0 Z+X+U\\
Z&= \sin X + V
\end{align*}
where $\theta_0=3$ and $U,V$ are independent standard Normal random variables. We have used kernel methods to estimate $E(Y \mid X)$ and $E(Z \mid X)$. We implement Algorithm \ref{AlgoPost} and derive the posterior distribution for the targeted parameter using 1000 Bayesian bootstrap samples, across 1000 replicate analysis. We report the results in Table \ref{ATEPLMkernel} of the Supplementary Material.
 \begin{table}
 \centering
	\caption{Characteristics of the frequentist and Bayesian distributions\break}
    {  \begin{tabular}{@{}lrrrr}
       \hline
       & \multicolumn{3}{c}{$n$} \\[3pt]
     &      \multicolumn{1}{c}{$250$}
& \multicolumn{1}{c}{$500$} & \multicolumn{1}{c}{$1000$}
\\
        \hline
		Average of the Posterior Means& $2.99$ & $3.00$ & $3.00$\\ 
		  Empirical Frequentist Mean& $2.99$ & $3.00$ & $3.00$ \\ 
		Average of Posterior Variances $(\times n)$& $1.00$ & $1.00$ & $0.99$ &\\ 
		Empirical Frequentist Variance $(\times n$) & $1.28$ & $1.19$ & $1.07$ \\ 
	    Average Sandwich Estimate   & $1.00$  & $1.00$  & $1.00$\\
       \hline
       Average Bayesian credible interval & $(2.87,3.12)$  &$(2.91,3.08)$ & $(2.94,3.06)$\\
       Frequentist confidence interval &$(2.86,3.14)$ &$(2.90,3.09)$ &$(2.94,3.06)$\\
       Posterior Coverage    & $92.90$  & $93.60$  & $94.50$ 
    \end{tabular}}
	\label{ATEPLMkernel}
	\label{tab2}
\end{table}
\subsection{Simulations for the AIPW estimator}{\label{SimAIPW}}
We now consider the AIPW estimator of the average treatment effect. The data generating mechanism is the same as the one used in subsection \ref{Sim} of the main paper. However, the estimation procedure differs as the fitted model involves estimating $\mu(Z,X)= \theta_0 Z +g_0(X)$, without assuming the additive relation between the treatment and the treatment-free part. We estimate $\mu_o(Z,X)$ and $e_0(X)$ using random-forests with sub-sampling without replacement with subsample size $m=n^{0.49}$. We implement Algorithm \ref{AlgoPost} and derive the posterior distribution for the targeted parameter using 1000 Bayesian bootstrap samples, across 1000 replicate analysis. We report the results in Table \ref{ATEAIPW} of the Supplementary Material.
\begin{table}
\centering
	\caption{Characteristics of the Bayesian and Frequentist Distributions\break  of the ATE based on the AIPW Estimator }
    {  \begin{tabular}{@{}lrrrr}
       \hline
       & \multicolumn{3}{c}{$n$} \\[3pt]
     &      \multicolumn{1}{c}{$250$}
& \multicolumn{1}{c}{$500$} & \multicolumn{1}{c}{$1000$}
\\
        \hline
		Average of the Posterior Means& $2.96$ & $2.99$ & $3.00$\\ 
		  Empirical Frequentist Mean& $2.96$ & $2.99$ & $3.00$ \\ 
		Average of Posterior Variances $(\times n)$& $5.11$ & $5.20$ & $5.09$ &\\ 
		Empirical Frequentist Variance $(\times n$) & $4.85$ & $5.08$ & $4.87$ \\ 
	    Average Sandwich Estimate   & $5.13$  & $5.21$  & $5.10$\\
       \hline
       Average Bayesian credible interval & $(2.68,3.24)$  &$(2.79,3.19)$ & $(2.86,3.14)$\\
       Frequentist confidence interval &$(2.69,3.22)$ &$(2.79,3.19)$ &$(2.86,3.14)$\\
       Posterior Coverage    & $95.20$  & $95.50$  & $95.0$ 
	\end{tabular}}
	\label{ATEAIPW}
	\label{tab2a}
\end{table}
\subsection{Impact of the number covariates}
Table \ref{FiniteSampleStudy} of the Supplementary Material shows the impact of the number of covariates on the Bayesian and frequentist estimation. We point out that the results presented in Section \ref{MainResults} are of asymptotic nature, and assume that the sample size $n \to \infty.$ In finite samples, we may observe bias which incites us to carefully examine whether the asymptotic arguments in Theorem \ref{maintheorem} can be deployed. In order to document this bias, we refer to Table \ref{FiniteSampleStudy} for a summary of the simulations for the model considered in subsection \ref{Sim} of the main paper for five different values of  $q \in \left\{5,6,8,10,20\right\}$ and for four different sample-sizes $n \in \left\{250,500,1000,2000 \right\}.$ We point out that $\tilde{\theta}_F$ and $\tilde{V}_F$ denote the empirical frequentist mean and variance (times $n$) respectively. $\tilde{\theta}_B$, $\tilde{V}_B$ are respectively the averages of the empirical posterior means and variances (times $n$). $\hat{\Sigma}$  denotes the averages of the sandwich estimates. It becomes apparent that for a given model, increasing the sample size reduces the bias as one may expect. However, it is remarkable that if $q=20$, bias is reduced but not eliminated even when $n=2000$. The variance estimates also exhibit heavy bias and coverage rates are below nominal level. These results reflect that in general, care must be taken before resorting to asymptotic arguments in Bayesian and frequentist semi-parametric theory, even when the nuisance space is assumed to be a Donsker class. Moreover, it may be the case that using a machine learning algorithm without theoretical guarantees on stochastic equicontinuity leads to good estimation, even when the number of covariates is high. However, theoretical guaranties on the asymptotic behavior or on the finite sample performance may not be available. Methods to restore posterior coverage at the nominal level, in the absence of the stochastic equicontinuity assumptions, are a subject of current studies by the authors and are not treated in this paper.
\begin{table}[]
    \centering
    \caption{Characteristics of the frequentist and Bayesian estimators for different values of $p$ and $n$.}
    \begin{tabular}{|c|c|c|c|c|c|}
    $p$ & $n$ & $(\tilde{\theta}_F,\tilde{\theta}_B)$ & $(\tilde{V}_F,\tilde{V}_B,\hat{\Sigma})$ & Coverage\\
    \hline
       \multirow{4}{3em}{$q=5$}  & $n=250$ &$(3.01,3.01)$ & $(5.08,5.36,5.38)$&$94.80$\\
       & $n=500$ & $(3.02,3.02)$&$(4.71,5.09,5.11)$ &$95.30$\\
        &$n=1000$& $(3.00,3.00)$&$(4.88,4.98,4.99)$ &$94.30$\\
        &$n=2000$ & $(3.00,3.00)$&$(4.97,4.91,4.90)$ &$94.40$\\
        \hline
    \hline
       \multirow{4}{3em}{$q=6$}  & $n=250$ &$(3.02,3.03)$ & $(4.98,5.61,5.63)$&$95.60$\\
       & $n=500$ & $(3.01,3.02)$& $(5.42,5.37,5.37)$&$93.60$\\
        &$n=1000$& $(3.00,3.00)$& $(4.78,5.21,5.22)$&$95.40$\\
        &$n=2000$ & $(3.00,3.00)$&$(4.82,5.07,5.09)$ &$95.20$\\
        \hline
          \hline
       \multirow{4}{3em}{$q=8$}  & $n=250$ &$(3.05,3.05)$ & $(5.30,6.03,6.07)$ &$95.30$\\
       & $n=500$ & $(3.02,3.03)$& $(5.43,5.83,5.84)$&$94.90$\\
        &$n=1000$& $(3.03,3.02)$& $(5.10,5.64,5.64)$&$95.20$\\
        &$n=2000$ & $(3.01,3.02)$& $(5.03,5.47,5.47)$&$93.80$\\
          \hline
          \hline
       \multirow{4}{3em}{$q=10$}  & $n=250$ &$(3.07,3.06)$ & $(5.91,6.32,6.34)$&$93.30$\\
       & $n=500$ & $(3.04,3.04)$& $(5.57,6.03,6.04)$ &$94.00$\\
        &$n=1000$& $(3.03,3.03)$& $(5.48,5.81,5.82),$&$93.10$\\
        &$n=2000$ & $(3.02,3.02)$& $(4.86,5.61,5.63)$&$94.80$\\
          \hline
          \hline
       \multirow{4}{3em}{$q=20$}  & $n=250$ &$(3.11,3.11)$ & $(6.06,7.18,7.22)$&$91.80$\\
       & $n=500$ & $(3.08,3.08)$& $(5.57,6.90,6.91)$&$91.40$\\
        &$n=1000$& $(3.06,3.06)$& $(5.81,6.56,6.58)$&$89.50$\\
        &$n=2000$ & $(3.05,3.06)$&$(5.69,6.33,6.33)$ &$87.80$\\
        \hline
        \hline
    \end{tabular}
    \label{FiniteSampleStudy}
\end{table}
\bibliographystyle{biometrika}
\bibliography{references}

\end{document}